\documentclass[12pt,letter]{article}
\usepackage{amssymb, amsthm}
\usepackage{amsfonts}
\usepackage{amsmath}
\usepackage{graphicx}
\usepackage[onehalfspacing]{setspace}
\usepackage{indentfirst}
\usepackage{longtable}
\usepackage{fancyhdr}
\usepackage{setspace}
\usepackage{multirow}

\usepackage{hyperref}

\newtheorem{theorem}{Theorem}
\newtheorem{lemma}{Lemma}
\newtheorem{corollary}{Corollary}

\begin{document}

\title{Nonparametric estimation of locally stationary Hawkes processes}

\author{E. Mammen\\ Institute of Applied Mathematics\\ Heidelberg University\\ Im Neuenheimer Feld 205\\
69120 Heidelberg, Germany\\ {mammen@math.uni-heidelberg.de}}

\maketitle

\begin{abstract} In this paper we consider multivariate Hawkes processes with baseline hazard and kernel functions that depend on time. This defines a class of locally stationary processes. We discuss estimation of the time-dependent baseline hazard and kernel functions based on a localized criterion. Theory on stationary Hawkes processes is extended to develop asymptotic theory for the estimator in the locally stationary model.

\end{abstract}
\newpage
\section{Introduction} A stationary mulitvariate Hawkes process is defined by the specification of a vector of baseline hazards and a matrix valued kernel function that model the input of recent jumps on the hazard. In this paper we generalize this to the case that baseline hazard and kernel function depend on time. We will use a locally stationary specification where the process can be applied by stationary Hawkes process, see \cite{D1996} for a general discussion of locally stationary processes. 
We will propose nonparametric  estimators  for baseline hazard and kernel function and we will develop asymptotic theory for these estimators. 

There is by now a rich literature on statistical inference for Hawkes processes. Hawkes processes have been applied in a grewing number of fields, including  crime analysis, see \cite{MSBST2011}, in statistical modeling of e-mail networks, see \cite {FSSCD2016}, and in genome analysis, see \cite{R-BS2010}.  A major number  of applications is coming from finance, for an overview see \cite{BMM2015}. A growing part of literature is concerned with nonparametric inference for Hawkes processes. 
For stationary nonparametric multivariate Hawkes processes, \cite{BDM2012} estimates the kernel of a multivariate nonparametric Hawkes process by estimating the jumps correlation matrix and using spectral methods. In \cite {BM2016} estimates of kernels of multivariate Hawkes processes are discussed that are based on solving empirical  integral equations with estimated average intensity vectors and estimated conditional laws. This approach is used in \cite{RBL2016} for the study of order book dynamics.
The papers \cite{Ki2016a} and \cite{Ki2016b} relate Hawkes processes to integer valued autoregressive processes of infinite order  INAR($\infty$). They approximate these processes by integer valued autoregressive processes of finite order $p< \infty$, INAR($p$), and use  methods  from statistical inference for INAR($p$) processes.  In \cite {EDD2017}
 a nonparametric estimator of kernels is proposed for multivariate Hawkes processes based on discretisations of the process. Its consistency is shown and the estimator is used for causal inference. The paper  \cite{BCC2015} discusses observations of $n$ independent Hawkes processes with constant baseline hazard and nonparametric kernel. The paper shows rates of convergence for a nonparametric maximum likelihood sieve estimator and it proves asymptotic normality for the parametric estimator of the baseline hazard. The papers \cite{JR-BR2015} and \cite{R-BS2010} develop deep theory for Hawkes processes and apply it for an asymptotic analysis of adaptive and LASSO-estimation of Hawkes processes.

There are some attempts to allow for nonstationary models. 
The papers 
\cite{ChHa2013} and \cite {ChHa2015} allow for a varying baseline hazard. They study estimates in an asymptotic framework where  the baseline hazard is multiplied by a factor that converges to infinity.  In the asymptotics the time horizont is kept fixed. The papers discuss models with parametric kernels and they allow for parametric and nonparametric specifications of the baseline hazard. In \cite {OHA2017} a Bayesian approach for models with time varying background rate is developed. The paper \cite{CP2017} considers parametric Hawkes processes with parameters depending on time.
 In \cite{RvSS2016}, \cite{RvSS2017} a new class of locally stationary multi-dimensional  Hawkes processes is proposed. Nonparametric estimation is discussed that is based on local Bartlett spectrums. The approach allows to compute approximations of first and second order moments.

In this paper we will  consider a model for multivariate Hawkes processes $N= (N^{(1)},...,N^{(d)})^\intercal$ where the  intensity function is defined by
\begin{equation} \label{definthawkes} \lambda^{(l)}(t) = \nu^{(l)}\left (\frac t T\right) + \sum_{m=1}^d\int_{t-A}^{t-} \mu^{(l,m)}\left (t-s;\frac t T\right) \ \mbox{d}N^{(m)}_s\end{equation}
for $l=1,...,d$, or in vector notation:
$$\lambda(t) = \nu\left (\frac t T\right) +\int_{t-A}^{t-} \mu\left (t-s;\frac t T\right)\ \mbox{d}N_s.$$
 We denote the true parameter functions  by $\nu_0= (\nu^{(m)}_0)_{m=1,...,d}$ and $\mu_0=(\mu^{(l,m)}_{0})_{l,m=1,...,d}$ with resulting intensity function $\lambda_0 =(\lambda_0^{(m)})_{m=1,...,d}$. We will use an asymptotic framework where $T \to \infty$.
We do not indicate in our notation that the process $N$ depends on $T$.  

In an alternative model one could consider that in \eqref{definthawkes} the term $\mu^{(l,m)} (t-s;$ $\frac t T)$ is replaced by $\mu^{(l,m)}\left (t-s;\frac sT\right)$. Estimators and an asymptotic theory for this model could be developed by the same approch as used in this paper for the model \eqref{definthawkes}. In this paper we only will discuss the specification \eqref{definthawkes} that as we expect is better motivated in more applications.

In this paper we will develop theory for estimation of the parameter $\nu^{(l)}_0\left (x_0\right)$ and the function $\mu^{(l,m)}_{0}\left (u;x_0\right)$ for $1 \leq m \leq d$ for a fixed value $x_0 \in (0,1)$ and for a fixed value of $l$. Without loss of generality we choose $l=1$ and we write 
$\nu^{*}_0 =\nu^{(1)}_0\left (x_0\right)$,  $\mu^{*,(m)}_{0}\left (u\right)=\mu^{(1,m)}_{0}\left (u;x_0\right)$, and $\mu^{*}_{0}\left (u\right)= (\mu^{*,(m)}_{0}\left (u\right))_{1\leq m \leq d}$.
We also write $t_0= x_0 T$. Note that this value depends on $T$. We assume that the Hawkes process is observed on an interval $[0,T]$. To simplify discussions we assume that the observed process has the intensity function \eqref{definthawkes} for $(- \infty, T]$. All counting processes considered in this paper are normed to be equal to 0 for $t=0$.

Our estimation strategy will be introduced in the next section. Asymptotic theory will be developed in Section \ref{asympsec}. Proofs are deferred to Section \ref{secproofs}.

\section {Estimation strategy}

We now come to the definition of our estimator. It is based on B-spline fits with accuracy measured by a local criterion function that is localized around $x_0$.
The estimator $\left (\hat \nu^* , \hat \mu^*\left (\cdot \right)\right )$ of $\left (\nu_0^* , \mu^*_0\left (\cdot \right)\right )$ is equal to
\begin{eqnarray*}
\hat \nu^* &=& \hat \theta_{0,1},\\
\hat \mu^*(u) &=& \sum_{j=1} ^J \hat \theta_{j,1} \psi_j(u),
\end{eqnarray*}
 where  $\psi_1, \psi_2, ..., \psi_J$ is a B-spline basis for $d$-dimensional functions on $[0,A]$ with equi-distant knot-points and with norm $\| \psi_l\|^2 = \int \psi^\intercal_l(x) \psi_l(x) \mathrm d x =1$ for $1 \leq l \leq J$ and where $\hat \theta_{0,1}, \hat \theta_{1,1}, ..., \hat \theta_{J,1} $ are defined as follows.

We choose $\hat \theta =(\hat \theta_0^\intercal,..., \hat \theta_J^\intercal) ^\intercal$ with $\hat \theta_j  = (\hat \theta_{j1},...,\hat \theta_{jK})^\intercal$  for $0 \leq j \leq J$ such that for some bandwidth $h \to 0$ and basis dimension $J \to \infty$
\begin{eqnarray*}
\rho(\theta) &=&- \frac 2 T \int \lambda^\#(t;  \theta)  h^{-1} K\left ( \frac {t-t_0} {Th}\right ) \mbox{d}N^{(1)}_t \\ && \qquad+  \frac 1 T \int \lambda^\# (t; \theta)^2 h^{-1} K\left ( \frac {t-t_0} {Th}\right ) \mbox{d}t 
\end{eqnarray*}
is minimized for $ \theta= \hat \theta$. The function $K$ is a kernel function, i.e. a probability density function. Furthermore, 
\begin{eqnarray*}\lambda^\# (t;  \theta) &=&  \sum_{k=1}^K \theta_{0,k} \left ( \frac {t-t_0}{Th}\right ) ^{k-1}\\ &&\qquad  + \sum_{j=1}^J  \sum_{k=1}^K \theta_{jk}  \int_{t-A}^{t-} \left ( \frac {t-t_0}{Th}\right ) ^{k-1} \psi_j(t-u)^\intercal \mbox{d}N_u.\end{eqnarray*}
Note that 
\begin{eqnarray*}
\rho(\theta) &=&- 2 \hat \tau^\intercal \theta + \theta^\intercal \Delta \theta, 
\end{eqnarray*}
where $\hat \tau = (\hat \tau_{01},..., \hat \tau_{0K},\hat \tau_{11},..., \hat \tau_{JK})^\intercal$ and where $$\Delta = \left(\begin{array}{cccc} \Delta_{00} & \Delta_{01} & ... &  \Delta_{0J}  \\ \Delta_{01}^\intercal   & \Delta_{11} & ... & \Delta_{1J} \\ \vdots & \vdots & \ddots & \vdots\\\ \Delta_{0J}^\intercal  & \Delta_{J1}  & ... & \Delta_{JJ}\end{array}\right)$$ with
 \begin{eqnarray*}
 \Delta_{0,0} &=&\left ( \frac 1 {Th} \int_{t_0 - Th}^{t_0 + Th} \left ( \frac {t-t_0}{Th}\right ) ^{k+k^\prime -2}    K\left ( \frac {t-t_0} {Th}\right )  \mbox{d} t\right )_{1\leq k, k\prime \leq K},\end{eqnarray*}
   \begin{eqnarray*}
  \Delta_{0,j} &=&\left ( \frac 1 {Th} \int_{t_0 - Th}^{t_0 + Th} \int_{t-A}^{t-} \psi^\intercal_j(t-u)   \left ( \frac {t-t_0}{Th}\right ) ^{k+k^\prime -2}    K\left ( \frac {t-t_0} {Th}\right )  \mbox{d} N_u \mbox{d} t\right )_{1\leq k, k\prime \leq K},
  \\
  \Delta_{j,j^\prime} &=&\left ( \frac 1 {Th} \int_{t_0 - Th}^{t_0 + Th} \int_{t-A}^{t-}\int_{t-A}^{t-} \psi^\intercal_j(t-u)   \mbox{d} N_u\  \psi^\intercal_{j^\prime}(t-v) \mbox{d} N_v\right . \\ && \qquad \times \left .   \left ( \frac {t-t_0}{Th}\right ) ^{k+k^\prime -2} K\left ( \frac {t-t_0} {Th}\right ) \mbox{d} t\right )_{1\leq k, k\prime \leq K}.
\end{eqnarray*}
Furthermore, we define $\hat \tau = \left(\hat \tau_0 ^\intercal , ...,\hat \tau_J ^\intercal \right)^\intercal $ with
 \begin{eqnarray*}
\hat \tau_0 &=& \left (\frac 1 {Th} \int_{t_0 - Th}^{t_0 + Th}  \left ( \frac {t-t_0}{Th}\right ) ^{k-1} K\left ( \frac {t-t_0} {Th}\right ) \mbox{d} N^{(1)}_t \right )_{1 \leq k \leq K}, \\
\hat \tau_j &=&\left ( \frac 1 {Th} \int_{t_0 - Th}^{t_0 + Th}\int_{t-A}^{t-} \psi^\intercal_j(t-u)   \left ( \frac {t-t_0}{Th}\right ) ^{k-1}K\left ( \frac {t-t_0} {Th}\right ) \mbox{d} N_u  \mbox{d} N_t^{(1)}\right )_{1 \leq k \leq K}.
\end{eqnarray*}
Note that 
$$\hat \theta = \Delta^{-1} \hat \tau$$
as long as $\Delta$ is invertible.

\section{Asymptotic analysis} \label{asympsec}

We make the following assumptions:

\begin{itemize}
\item[(A1)] The support of the functions $\mu_{0}^{(l,m)}(s,x)$ is contained in $[0,A]\times [0,1]$ for some $A>0$ and the functions are bounded and  positive for $l,m = 1,...,d$, $s\in [0,A], x \leq 1$. The matrix $\Gamma^+$ has spectral radius strictly smaller than 1. Here we define $\Gamma^+$ as the matrix with elements $$\int _0 ^A \sup_{x \leq 1}\mu_{0}^{(l,m)} (s, x) ds .$$ The function $\nu_0(x)$ is bounded and bounded away from 0 for $x \leq 1$. 
\item[(A2)] The partial derivatives of  $\mu_{0}^{(l,m)}(s,\cdot)$ with respect to the second argument exist  and they are uniformly absolutely bounded  for $l,m = 1,...,d$ and $0 \leq s \leq A$.

\item[(A4)] It holds that $J(\log T)^5  \frac 1 {\sqrt{hT}}\to 0$, $\log T /J \to 0$ and $h T^{1-\delta} \to \infty$ for $\delta > 0$ small enough. 

\item[(A5)] There exist $\theta^*_{j,k}$ for $k=1,..,K$ and $j=0,...,J$, depending on $J$ such that \begin{eqnarray*} &&\left | \nu^{(1)}(x)- \sum_{k=1}^K \theta^*_{0,k}  (x-x_0)^{k-1} \right | \leq \varepsilon_T,\\
&& \left | \mu_0^{(1,m)}(u,x)- \sum_{k=1}^K \sum_{j=1}^J \theta^*_{j,k} \psi_j(u) (x-x_0)^{k-1} \right | \leq \varepsilon_T
\end{eqnarray*}
for some sequence $\varepsilon_T \to 0$ and for $u \in[ 0,A]$ and $|x-x_0| \leq h$.

\end{itemize}

\begin{theorem}\label{theo1} Make Assumptions (A1) -- (A5). It holds that
\begin{eqnarray*}
\hat \nu^* - \nu_0^* &=& O_P\left (\varepsilon_T + \sqrt{\frac J {hT}}\right ),\\
\left ( \int {[ \hat \mu^{*,(l)}(u) - \mu_0^{*,(l)}(u)]^2 \mathrm{d} u} \right)^{1/2}&=& O_P\left (\varepsilon_T + \sqrt{\frac J {hT}}\right )
\end{eqnarray*}
for $l=1,...,d$.
\end{theorem}

By using some simplifications in our proofs one gets the following result for a stationary Hawkes process $N_t$ with intensity function
\begin{eqnarray} \label{stat}\lambda^{(l)}(t) = \nu^{(l)} + \sum_{m=1}^d\int_{t-A}^{t-} \mu^{(l,m)}\left (t-s\right) \mathrm{d} N_s^{(m)}\end{eqnarray}
 for some constants $\nu^{(l)}$  and functions $\mu^{(l,m)}\left (\cdot \right)$ $(1 \leq l,m \leq d)$.
Now, the estimator $\left (\hat \nu , \hat \mu\left (\cdot \right)\right )$ of $\left (\nu_0 , \mu_0\left (\cdot \right)\right )$ is defined by $
\hat \mu(u) = \sum_{j=1} ^J \hat \beta_{j} \psi_j(u),
$
 where  $(\hat \nu, \hat \beta_{1}, ..., \hat \beta_{J})$ minimizes
\begin{eqnarray*}
- \frac 2 T \int \lambda^\#(t; \nu, \beta)   \mbox{d}N^{(1)}_t +  \frac 1 T \int \lambda^\# (t;\nu, \beta)^2  \mbox{d}t 
\end{eqnarray*}
with $\lambda^\# (t; \nu, \beta) =   \nu + \sum_{j=1}^J   \beta_{j}  \int_{t-A}^{t-}  \psi_j(t-u)^\intercal \mbox{d}N_u$. In the stationary case we get the following result:
\begin{corollary} For a stationary process with \eqref{stat}
assume that for $1 \leq l,m \leq d$
the support of the functions $\mu_{0}^{(l,m)}(s)$ is contained in $[0,A]$ for some $A>0$, that the functions are positive for $l,m = 1,...,d$, $s\in [0,A]$, and that the matrix $\Gamma^+$ has spectral radius strictly smaller than 1. Now, we define $\Gamma^+$ as the matrix with elements $\int _0 ^A \mu_{0}^{(l,m)} (s) ds $. Furthermore, we assume that $\nu_0 >0$, that $J(\log T)^5  /\sqrt {T}\to 0$ and that
there exist $\beta^*_{j}$ for  $j=1,...,J$, depending on $J$, such that \begin{eqnarray*} 
&& \left | \mu_0^{(1,m)}(u)-  \sum_{j=1}^J \beta^*_{j} \psi_j(u)  \right | \leq \varepsilon_T
\end{eqnarray*}
for some sequence $\varepsilon_T \to 0$ and for $u \in[ 0,A]$. 
Then,  it holds that $\hat \nu - \nu_0 = O_P\left (\varepsilon_T + \sqrt{J/T}\right )$ and
\begin{eqnarray*}
\left ( \int {[ \hat \mu^{(1,m)}(u) - \mu_0^{(1,m)}(u)]^2 \mathrm{d} u} \right)^{1/2}&=& O_P\left (\varepsilon_T + \sqrt{J/T}\right )
\end{eqnarray*}
for $m=1,...,d$.
\end{corollary}
Up to an additional log factor this result can also be proved along the lines of arguments used in 
\cite {JR-BR2015}. There an additional log factor appears because adaptive LASSO estimation is considered. 

\section{Proofs} \label{secproofs}

In our proofs the quantities $C, C^*, C_1,...$ are positive constants that are chosen large enough and $c, c^*, c_1,...$ are strictly positve constants that are chosen small  enough. The same value names will be used for different constants, also in the same formula.   For simplicity 
we assume  that $K=1$. All arguments go through for $K>0$ at the cost of a more complex notation. Then, our estimator $\left (\hat \nu^* , \hat \mu^*\left (\cdot \right)\right )$ of $\left (\nu_0^* , \mu^*_0\left (\cdot \right)\right )$ is defined as 
\begin{eqnarray*}
\hat \mu^*(u) = \sum_{j=1} ^J \hat \beta_j \psi_j(u),
\end{eqnarray*}
where for  $\hat \beta  = (\hat \beta_{0},...,\hat \beta_{J})^\intercal$ with some bandwidth $h \to 0$ and basis dimension $J \to \infty$
\begin{eqnarray*}
\rho(\beta) &=&- \frac 2 T \int \lambda^\#(t;  \beta)  h^{-1} K\left ( \frac {t-t_0} {Th}\right ) \mbox{d}N^{(1)}_t \\
&& \qquad+  \frac 1 T \int \lambda^\# (t; \beta)^2 h^{-1} K\left ( \frac {t-t_0} {Th}\right ) \mbox{d}t 
\end{eqnarray*}
is minimized. Here, as above,  $\psi_1, ..., \psi_J$ is a B-spline basis for $d$-dimensional functions on $[0,A]$ and $\lambda^\# (t;  \beta) =  \beta_0 + \sum_{j=1}^J  \beta_j \int_{t-A}^{t-} \psi_j(t-u)^\intercal \mbox{d}N_u$.
Note that now
$
\rho(\beta) =- 2 \hat \tau^\intercal \beta + \beta^\intercal \Delta \beta$,
where $$\Delta = \left(\begin{array}{cc}1 & \delta^\intercal \\ \delta & \Delta^*\end{array}\right)$$ with
 \begin{eqnarray*}
\delta&=&\left ( \frac 1 {Th} \int_{t_0 - Th}^{t_0 + Th}\int_{t-A}^{t-} \psi^\intercal_j(t-u)  K\left ( \frac {t-t_0} {Th}\right ) \mbox{d} N_u \mbox{d} t\right )_{j=1}^J,  \\
\Delta^* &=&\left ( \frac 1 {Th} \int_{t_0 - Th}^{t_0 + Th}\int_{t-A}^{t-}\int_{t-A}^{t-} K\left ( \frac {t-t_0} {Th}\right ) \psi^\intercal_j(t-u)   \mbox{d} N_u\  \psi^\intercal_k(t-v) \mbox{d} N_v \mbox{d} t\right )_{k,j=1}^J.
\end{eqnarray*}
and, where 
$\hat \tau = \left(\hat \tau_1, \hat \tau_2\right)^\intercal$ with
 \begin{eqnarray*}
\hat \tau_1 &=& \frac 1 {Th} \int_{t_0 - Th}^{t_0 + Th}  K\left ( \frac {t-t_0} {Th}\right ) \mbox{d} N^{(1)}_t, \\
\hat \tau_2 &=&\left ( \frac 1 {Th} \int_{t_0 - Th}^{t_0 + Th}\int_{t-A}^{t-} \psi^\intercal_j(t-u)  K\left ( \frac {t-t_0} {Th}\right ) \mbox{d} N_u  \mbox{d} N_t^{(1)}\right )_{j=1}^J.
\end{eqnarray*}Note that 
$$\hat \beta = \Delta^{-1} \hat \tau$$
as long as $\Delta$ is invertible.

Our proofs make use of the cluster representations of Hawkes processes, see \cite{HO1974} and \cite{R-BR2007}. This representation has also been used in the statistical analysis of nonparametric estimators of stationary Hawkes processes in \cite{JR-BR2015} and of locally  stationary Hawkes processes in \cite{RvSS2016}, \cite{RvSS2017} whereas the definition of locally  stationary Hawkes processes introduced in \cite{RvSS2016}, \cite{RvSS2017} differs from the notion studied here. We now define a cluster representation for our non-stationary inhomogeneous process $N$.  The cluster representation for our Hawkes process \eqref{definthawkes} is given by the following independent random variables: $P_x^{(l,m)}$ ($x \in [0,1], 1 \leq l,m \leq d$) that have a Poisson distribution with parameter $p_x^{(l,m)}= \int_0^A \mu^{(l,m)}(s,x+ (s/T)) \mbox{d}s$; $P^{(m)}$ ($1 \leq m \leq d$) that have a Poisson distribution with parameter $p^{(m)}= T \int_0^1 \nu^{(m)}(s) \mbox{d}s$; $X_{x,i}^{(l,m)}$  ($x \in [0,1], 1 \leq l,m \leq d$, $i \in \mathbb{N}$) that have a Lebesgue density $\mu^{(l,m)}(\cdot,x+ (\cdot/T))/p^{(l,m)}_x$; and $Z_i^{(m)}$ ($1 \leq m \leq d$, $i \in \mathbb{N}$) that have a Lebesgue density $\nu^{(m)}(\cdot/T)/p^{(m)}$. The construction of the Hawkes process now starts with the definition of birth dates of socalled immigrants or ancestors. We assume that there are ancestors of type $1 \leq m \leq d$. The birth dates of ancestors of type $m$ are given by a Poisson process with intensity $\nu^{(m)}(\cdot/T)$ or equivalently, as $Z_1^{(m)},...,Z_{P^{(m)}}^{(m)}$. Each ancestor sets up a separate family with descendants of type $1 \leq l \leq d$. The descendants are born iteratively in generations $n \geq 1$. A member of a family that was born at time $S$ and that is of type $m$ has  $P_{S/T}^{(l,m)}$ children of type $l$. They are born at the dates $S+ X_{S/T,1}^{(l,m)}, ..., S+ X_{S/T,P_{S/T}^{(l,m)}}^{(l,m)}$. In the construction we made use of the fact that with probability 1 no two individuals are born at the same date. This follows because our intensity measures are assumed to have Lebesgue densities. In particular, for this reason none of the variables 
$P_x^{(l,m)}$ ($x \in [0,1], 1 \leq l,m \leq d$) and $X_{x,i}^{(l,m)}$  ($x \in [0,1], 1 \leq l,m \leq d$, $i \in \mathbb{N}$) is used twice or more times in the construction, with probability equal to 1.

The cluster construction allows to compare the process $N$ with the homogeneous Hawkes process $\bar N$ that has intensity function
\begin{equation} \label{definthawkessup} \bar \lambda^{(m)}(t) = \bar \nu^{(m)}+ \sum_{l=1}^d\int_{t-A}^{t-} \bar \mu^{(l,m)}\left (t-u\right) \ \mbox{d}\bar N^{(m)}_s\end{equation}
with $\bar \nu^{(m)}= \sup_{x\leq 1} \nu^{(m)}\left (x\right)$ and $\bar \mu^{(l,m)}\left (t\right) = \sup_{x\leq 1} \mu^{(l,m)}\left (t;x\right) $ for $1 \leq l,m \leq d$ and $0 \leq t \leq A$.
By the arguments above, it can be  seen that there exists a strong construction of $N$ and $\bar N$ such that all birth points of $N^{(l)}$ are also birth points of $\bar N^{(l)}$ for $1 \leq l \leq d$. Thus we can carry over the results  of Lemma 1 and of Proposition 2 in \cite {JR-BR2015} that treat homogeneous Hawkes processes to our inhomogeneous Hawkes process and we get the following result.

\begin{lemma} \label{lem1} Make the assumptions of Theorem \ref{theo1}.  There exist $\rho > 0$ and $C> 0$, not depending on $T$, such that
\begin{eqnarray}
\label{boundnumber}
&&\mathbb{E} \left [ e^{\rho W_l }\right ] < C,\\
\label{boundnumber2}
&&\mathbb{E} \left [ e^{\rho N_{[t-A,t]}}\right ] < C
\end{eqnarray}
for $1 \leq l \leq d$ and $t \leq T$. Here $W_l$ is the number of family members in a family with ancestor of type $l$ and $N_{[t-A,t]}$ is the number of points of all types of $N$ in the interval $[t-A,t]$. 
\end{lemma} 

At this point we would also like to add another result that  follows from the theory of homogeneous Hawkes processes and 
that will be used in the local mathematical analysis of our estimator $\hat \theta$. Denote by 
$T^*_{e,t}$ the last birth date of all types inside all families whose ancestor was born before $t$ and put $T_{e,t}=T^*_{e,t}-t$. This is also called extinction time. For the homogeneous process $\bar N$ the distribution of $T_{e,t}$ does not depend on $t$ and it holds that $\mathbb{P} [T_{e,t} \geq s] \leq \sum _{l=1} ^d \frac { \bar \nu^{(l)}}{ \rho_l} \mathbb{E} \left [ e^{\rho_l W_l }\right ] e ^{-\rho_l s}$, see the proof of Proposition 3 in  \cite {JR-BR2015}. Again using the above strong approximation, we can carry over this result to our  inhomogeneous Hawkes process and we get the following lemma. 
\begin{lemma} \label{lem2} Make the assumptions of Theorem \ref{theo1}. We have for some constants $C, \rho^* > 0$ that do not depend on $T$
\begin{eqnarray}
\label{boundnumber3}
&&\mathbb{P} [T_{e,t} \geq s] \leq C e ^{-\rho^* s}\end{eqnarray}
for $t \leq T$ and $s \geq 0$.
\end{lemma}

For the study of the terms $\delta$ and $\Delta$ we will use that these quantities can be approximated by sums of independent random variables. For this aim we will use a construction that also has been used in  \cite {JR-BR2015} and \cite{R-BR2007}  for the study of homogeneous Hawkes processes. For $x^*, x $ fixed, suppose  that $N_{q,n}$ ($q \in \mathbb{N}, n \in \mathbb{Z}$) are independent Hawkes processes with intensity function
\begin{eqnarray*} \lambda_{q,n}^{(l)}(t)&=&  \lambda_{n}^{(l)}(t) = \nu^{(m)}\left (\frac t T\right) \mathbb{I}_{[x^*+2nx-x\leq t <x^*+ 2nx+x]} \\ && \qquad + \sum_{m=1}^d\int_{t-A}^{t-} \mu^{(l,m)}\left (t-s;\frac t T\right) \ \mbox{d}N^{(m)}_{q,n;s}\end{eqnarray*}
for $ 1 \leq l \leq d$.These are Hawkes processes that only contain families with ancestors born in the interval $[x^*+2nx-x, x^*+2nx+x)$. Put $N^+= \sum_{-\infty < n < \infty} N_{0,n}$ and $N_q= \sum_{-\infty < n < q-1} N_{q,n} + N_{0,q}$. It holds:
The Hawkes process $N^+$ has the same distribution as $N$. The ancestors of the Hawkes process $N_q$ are all born before $x^*+ 2qx+x$ and the process $N_q$ has intensity function 

\begin{eqnarray}
\label{hawkesstrong1}  \nu^{(l)}\left (\frac t T\right) \mathbb{I}_{[- \infty <  t <x^*+ 2qx+x]} + \sum_{m=1}^d\int_{t-A}^{t-} \mu^{(l,m)}\left (t-s;\frac t T\right) \ \mbox{d}N^{(m)}_{q;s}\end{eqnarray}
for $ 1 \leq l \leq d$. Furthermore, it holds that the processes $N_q$ ($q \geq 1$) are independent. We now put for $a,x$ fixed with $0< a < x$ \begin{eqnarray}
\label{hawkesstrong2}M_q^{x^*,x}= N_q\big | _{[x^*+ 2qx-a, x^*+2qx+x)}.\end{eqnarray}
These are independent processes for $q \in \mathbb{N}$. 
With the same arguments as in Section 3.1 of \cite{R-BR2007} for one-dimensional  homogeneous Hawkes processes one gets that 
\begin{eqnarray*}
\left | \mathbb{P} [M_q^{x^*,x} \not = N^+ \big | _{[x^*+2qx-a,x^*+ 2qx+x)} \in {\cal A}]\right |&\leq& 2 \mathbb{P} [T_{e,x^*+2qx-x} \geq x-a]\\
\nonumber & \leq& 2 C e ^{-\rho^* (x-a)}, \end{eqnarray*}
where \eqref{boundnumber3} has been used. 
In particular, we have that for measurable sets $\cal A$
\begin{eqnarray*}
\left | \mathbb{P} [M_q^{x^*,x} \in {\cal A}]-  \mathbb{P} [N\big | _{[x^*+2qx-a,x^*+ 2qx+x)} \in {\cal A}]\right |&\leq&  2 C e ^{-\rho^* (x-a)}. \end{eqnarray*}
By a small extension of the arguments one gets the following lemma.

\begin{lemma} \label{lem3} Make the assumptions of Theorem \ref{theo1}. There exists a Hawkes process $N^+$ that has the same distribution as $N$ with the following property. For a finite subset $I$ of $\mathbb{N}$, $x^*\in [0,1)$, and $x >a > 0$ there exist independent Hawkes processes $N_q$ for $q \in I$ with intensity function \eqref{hawkesstrong1} such that for the processes $M_q^{x^*,x}$ defined in \eqref{hawkesstrong2} it holds that \begin{eqnarray}
\label{boundnumber4} \mathbb{P} [M_q^{x^*,x} \not =N^+\big | _{[x^*+2qx-a,x^*+ 2qx+x)}\ \mathrm{ for}\ \mathrm{ some }\  q \in I]&\leq& 2 C |I| e ^{-\rho^* (x-a)}.\end{eqnarray}
\end{lemma} 

We will use the following lemma for the calculation of second order moments of linear statistics of Hawkes processes of the form \eqref{definthawkes}. This lemma generalises results obtained in \cite{BDM2012} for 
stationary homogeneous Hawkes processes.

\begin{lemma} \label{lem4} For a  Hawkes processes $N_t$ of the form \eqref{definthawkes} that fulfils Assumption (A1), it holds for $\Lambda(t/T) =  \mathbb{E}\left [ \lambda(t)\right ]$ that:
\begin{eqnarray} \label{boundnumber5aa}
\Lambda\left (\frac t T\right ) &= &  \nu\left (\frac t T\right )  +\int \chi\left (t-s,\frac t T\right )  \nu\left (\frac s T\right )  \mathrm{d} s, \\
 \label{boundnumber5ab}
\lambda(t) &= & \Lambda\left (\frac t T\right ) + \int \chi\left (t-s,\frac t T\right )  \mathrm{d} M_s,\end{eqnarray}
where 
$$  \chi\left (t-s,\frac t T\right ) = \sum_{k=1}^\infty \mu^{(*k)}\left (t-s,\frac t T\right )$$ with $\mu^{(*1)}\left (t-s,\frac s T\right )  = \mu\left (t-s,\frac s T\right ) $ and 
$$\mu^{(*k)}\left (t-s,\frac t T\right ) = \int \mu^{(*(k-1))}\left (t-u,\frac t T\right )\mu \left (u-s,\frac u T\right ) \mathrm {d} u$$
for $k \geq 2$, and  where $M_t$ is the martingale defined by $\mathrm{d} M_t = \mathrm{d} N_t - \lambda(t)  \mathrm{d} t$.
Finally, we have that
\begin{eqnarray} \label{boundnumber5ac}
&&\mathbb{E}\left [\mathrm{d} N_t \mathrm{d} N_{t^\prime}^\intercal  \right]
= \left ( \Lambda\left (\frac t T\right ) \Lambda\left (\frac { t ^\prime} T\right )^\intercal + \Sigma_{t/T}\delta_{t-t^\prime} + \chi \left (t-t^\prime, \frac {t} T \right ) \Sigma _{t^\prime/T}   \right . \\ \nonumber 
&& \qquad \left .+ \Sigma _{t/T} \chi \left (t^\prime-t, \frac {t^\prime} T \right )^\intercal + \int  \chi \left (t-s, \frac {t} T \right ) \Sigma_{s/T}  \chi \left (t^\prime-s, \frac {t^\prime} T \right )^\intercal \mathrm{d} s\right )\mathrm{d}t \mathrm{d} t^\prime,
\end{eqnarray}
where $\Sigma_{t/T}$ is a diagonal matrix with diagonal elements $\Lambda_i(t/T)$.
\end{lemma} 
\begin{proof} Using the hazard defined in \eqref{definthawkessup}  we can  construct a stationary homogeneous process $\bar N_t$ with the property that all jumps of $N_t$ are also jumps of $\bar N_t$. In particular, this implies existence of $N_t$. Note that

\begin{eqnarray*} \lambda(t) &=& \nu\left (\frac t T\right )+  \int \mu \left (t-u, \frac t T\right )\mathrm {d} N_u\\
&=& \nu\left (\frac t T\right )+  \int \mu \left (t-u, \frac t T\right ) \lambda(u) \mathrm {d} u +  \int \mu \left (t-u, \frac t T\right )\mathrm {d} M_u
\\
&=& \nu\left (\frac t T\right )+  \int \mu \left (t-u, \frac t T\right ) \nu\left (\frac u T\right ) \mathrm {d} u +
 \int \mu^{(*2)} \left (t-u, \frac t T\right ) \lambda(u) \mathrm {d} u\\ && \qquad +  \int \sum_{l=1} ^2 \mu^{(*l)} \left (t-u, \frac t T\right )\mathrm {d} M_u
 \\
&=& \nu\left (\frac t T\right )+  \int  \sum_{l=1} ^{k-1} \mu^{(*l)} \left (t-u, \frac t T\right ) \nu\left (\frac u T\right ) \mathrm {d} u\\ && \qquad  +
 \int \mu^{(*k)} \left (t-u, \frac t T\right ) \lambda(u) \mathrm {d} u +  \int \sum_{l=1} ^k \mu^{(*l)} \left (t-u, \frac t T\right )\mathrm {d} M_u
\end{eqnarray*}
for $k\geq 1$. Because $\lambda$ and $ \mu^{(*l)} $ are positive and can be bounded by $\bar \lambda$ and $\bar  \mu^{(*l)} $ where $\bar  \mu^{(*l)} $  is defined as $  \mu^{(*l)} $ but with $N_t$ replaced by $\bar N_t$ we can conclude that the right hand side of the last equation converges in L$_2$ to 
$$ \nu\left (\frac t T\right )+  \int  \chi \left (t-u, \frac t T\right ) \nu\left (\frac u T\right ) \mathrm {d} u 
 +  \int \chi \left (t-u, \frac t T\right )\mathrm {d} M_u.$$
 Thus $\lambda(t) $ is equal to this expression and we conclude \eqref{boundnumber5aa} by taking the expectation of this expression. This also implies \eqref{boundnumber5ab}.
 
 For the proof of \eqref{boundnumber5ac} one proceeds similarly as in the proof of Proposition 2 in  \cite{BDM2012}. Note that: $$\mathbb{E}\left [\mathrm{d} N_t \mathrm{d} N_{t^\prime}^\intercal  \right] = I_1+...+I_4$$
 with \begin{eqnarray*} I_1= \mathbb{E}\left [\mathrm{d} M_t \mathrm{d} M_{t^\prime}^\intercal  \right] 
 = \Sigma_{t/T} \delta_{t-t^\prime} \mathrm{d}t \mathrm{d} t^\prime,
 \end{eqnarray*}
  \begin{eqnarray*} I_2&=& \mathbb{E}\left [\lambda(t) \mathrm{d} M_{t^\prime}^\intercal  \right] \mathrm{d} t\\
 &=& \mathbb{E}\left [\left\{ \Lambda\left (\frac t T\right )+ \int \chi\left ( t-s, \frac t T\right ) \mathrm{d} M_s\right \} \mathrm{d} M_{t^\prime}^\intercal  \right] \mathrm{d} t \\
 &=&  \int \chi\left ( t-s, \frac t T\right )\Sigma_s \delta_{s-t^\prime}\mathrm{d} s   \mathrm{d}t \mathrm{d} t^\prime \\
 &=&  \chi\left ( t-t^\prime, \frac {t} T\right )\Sigma_{t^\prime}   \mathrm{d}t \mathrm{d} t^\prime  ,\\ I_3 &=& \Sigma_{t} \chi\left ( t^\prime-t, \frac {t^\prime} T\right ) ^\intercal  \mathrm{d}t \mathrm{d} t^\prime  ,\\
 I_4&=& 
 \mathbb{E}\left [\lambda(t)\lambda (t^\prime)^\intercal \right ]  \mathrm{d}t \mathrm{d} t^\prime \\
 &=& 
 \mathbb{E}\left [\left\{ \Lambda\left (\frac t T\right ) + \int \chi\left (t-s,\frac t T\right ) \mathrm{d} M_s\right \}\lambda (t^\prime)^\intercal \right ]  \mathrm{d}t \mathrm{d} t^\prime \\
 &=& \Lambda\left (\frac t T\right ) \Lambda\left (\frac {t^\prime} T\right )^\intercal  \mathrm{d}t \mathrm{d} t^\prime \\
 && \qquad +  \mathbb{E}\left [\left\{ \int \chi\left (t-s,\frac t T\right ) \mathrm{d} M_s\right \}\left\{ \int \chi\left (t^\prime -u,\frac {t^\prime} T\right )^\intercal  \mathrm{d} M_u\right \} \right ]  \mathrm{d}t \mathrm{d} t^\prime \\
 &=& \Lambda\left (\frac t T\right ) \Lambda\left (\frac {t^\prime}T\right )^\intercal  \mathrm{d}t \mathrm{d} t^\prime + \int \chi\left (t-s,\frac t T\right ) \Sigma_{s/T}\chi\left (t^\prime -s,\frac {t^\prime} T\right )^\intercal  \mathrm{d} s  \mathrm{d}t \mathrm{d} t^\prime.
  \end{eqnarray*}
This shows equation \eqref{boundnumber5ac}.
\end{proof}

By application of the last lemma we get the following result. 

\begin{lemma} \label{lem5} Make the assumptions of Theorem \ref{theo1}. For functions $g: [0,A] \to \mathbb{R}^d$ and $\nu \in \mathbb{R}^d$ it holds that:
\begin{eqnarray} \label{boundnumber5newa}
R(g,\nu)  
&= &\frac 1 {Th} \int_{t_0 -Th}^{t_0+Th}  \left \|\nu + \int_{t-A}^{t-}g^\intercal(t-u)\Lambda\left (\frac  u T\right) \mathrm{d} u \right \|^2K\left ( \frac {t-t_0} {Th}\right ) \mathrm{d} t \\ \nonumber 
&& \quad + \frac 1 {Th} \int_{t_0 -Th}^{t_0+Th} \int \left ( g (t-v) + \int  \chi\left ( u-v, \frac u T\right)^\intercal g (t-u)\mathrm {d} u \right )^\intercal\\ \nonumber 
&& \qquad   \Sigma_{v/T}  \left ( g^\intercal (t-v) + \int  \chi\left ( u-v, \frac u T\right) ^\intercal g (t-u)\mathrm {d} u \right )\\ && \qquad \times K\left ( \frac {t-t_0} {Th}\right ) \mathrm d v \mathrm d t \nonumber \end{eqnarray}
for some constant $c>0$ where $$R(g,\nu)  = \frac 1 {Th} \mathbb{E}\left [\int_{t_0 - Th}^{t_0 + Th} \left (\nu + \int_{t-A}^{t-}g^\intercal(t-u)   \mbox{d} N_u \right )^2 K\left ( \frac {t-t_0} {Th}\right )\mbox{d} t\right ].$$
\end{lemma} 

We will use this lemma to prove the following result.

\begin{lemma} \label{lem6} Make the assumptions of Theorem \ref{theo1}. For functions $g: [0,A] \to \mathbb{R}^d$ and $\nu \in \mathbb{R}_+$ it holds that:
\begin{eqnarray} \label{boundnumber5newb}
R(g,\nu)  
&\geq &c \left (\|\nu\|^2 + \int_{0}^{A}g^\intercal(t)g(t) \mbox{d} t\right )
\end{eqnarray}
for some constant $c>0$. In particular, it holds that the smallest eigen value of $\mathbb{E}[\Delta] $ is bounded from below and that, thus,  the matrix $\mathbb{E}[\Delta] $ is invertible.\end{lemma}

\begin{proof}
For the lemma it suffices to show that 
\begin{eqnarray} 
\label{lowb1}
&& \frac 1 {Th} \int_{t_0 -Th}^{t_0+Th} \int \left ( g (t-v) + \int  \chi\left ( u-v, \frac u T\right)^\intercal g (t-u)\mathrm {d} u \right )^\intercal\\ \nonumber 
&& \qquad   \Sigma_{v/T}  \left ( g^\intercal (t-v) + \int  \chi\left ( u-v, \frac u T\right) ^\intercal g (t-u)\mathrm {d} u \right ) K\left ( \frac {t-t_0} {Th}\right )\mathrm d v \mathrm d t\\&& \geq c  \int_{0}^{A}g^\intercal(t)g(t) \mbox{d} t \nonumber
\end{eqnarray}
for some constant $c>0$. For the proof of this claim note that, because of (A2), with $c>0$ small enough the left hand side of \eqref{lowb1} can be bounded from below by 
\begin{eqnarray*} 
&&c \frac 1 {Th} \int_{t_0 -Th}^{t_0+Th} \int \left ( g (w) + \int  \chi\left ( w-s, \frac {t-s} T\right)^\intercal g (s)\mathrm {d} s \right )^\intercal\\
&& \qquad     \left ( g^\intercal (w) + \int  \chi\left ( w-s, \frac {t-s} T\right) ^\intercal g (s)\mathrm {d} s \right ) K\left ( \frac {t-t_0} {Th}\right )\mathrm d w \mathrm d t \end{eqnarray*}
   \begin{eqnarray*}
 && \geq c \frac 1 {Th} \int_{t_0 -Th}^{t_0+Th} \int \left ( g (w) + \int  \chi_0\left ( w-s\right)^\intercal g (s)\mathrm {d} s \right )^\intercal\\
&& \qquad     \left ( g^\intercal (w) + \int  \chi_0\left ( w-s\right) ^\intercal g (s)\mathrm {d} s \right ) K\left ( \frac {t-t_0} {Th}\right )\mathrm d w \mathrm d t- O(h)
\\
 && = c \int \left ( g (w) + \int  \chi_0\left ( w-s\right)^\intercal g (s)\mathrm {d} s \right )^\intercal\\
&& \qquad     \left ( g(w) + \int  \chi_0\left ( w-s\right) ^\intercal g (s)\mathrm {d} s \right ) \mathrm d w- O(h)\\
 && = c \int g (w)^\intercal g (w)  \mathrm d w +\int  \int g (w)^\intercal  \sum_{k=1}^\infty \mu^{*,(*k)}\left (w-s\right ) g (s)\mathrm {d} s\mathrm d w  \\ && \qquad +\int  \int g (w)^\intercal  \sum_{k=1}^\infty \mu^{*,(*2k)}\left (w-s\right ) g (s)\mathrm {d} s\mathrm d w  - O(h)
  \\
 && \geq c_1 \int g (w)^\intercal g (w)  \mathrm d w   - O(h)\end{eqnarray*}
for $c_1 >0 $ small enough, where 
$$  \chi_0\left (t-s\right ) = \sum_{k=1}^\infty \mu_0^{*,(*k)}\left (t-s\right )$$ with $\mu_0^{*,(*1)}\left (t-s,\right )  = \mu_0^*\left (t-s\right ) $ and 
$$\mu_0^{*,(*k)}\left (t-s\right ) = \int \mu_0^{*,(*(k-1))}\left (t-u\right )\mu_0^* \left (u-s\right ) \mathrm {d} u$$
for $k \geq 2$. Here, we used (A1) to get the last inequality. This shows \eqref{lowb1} and concludes the proof of the lemma.
\end{proof}

We now consider the variables $\delta, \Delta^*, \hat \tau_1$ and $\hat \tau_2$. With the help of Lemma \ref{lem3} with $a=A$ we can approximate these variables by the sum of two terms where each term is the sum of independent variables. Such a splitting device has also been used in \cite{R-BR2007} to prove Hoeffding and Bernstein inequalities for averages of flows induced by stationary Hawkes processes. We start by discussing $\Delta^*$. With $G(t) =(G_{jk}(t))_{1\leq j,k \leq d}$  where
$$G_{jk}(t) = \int_{t-A}^{t-} \int_{t-A}^{t-} K\left ( \frac {t-t_0} {Th}\right ) \psi_j^\intercal (t-u) \mathrm{d} N_u 
 \psi_k^\intercal (t-v) \mathrm{d} N_v \mathbb I_{(t_0 -Th \leq t \leq t_0 + Th)}$$
 we get that \begin{eqnarray*} 
 \Delta ^* - \mathbb{E} [ \Delta ^*]  &=& \frac 1 {Th} \int_{t_0 - Th}^{t_0 + Th} \left \{ G(t) - \mathbb{E} [ G(t)]\right \} \mathrm d t \\&=& \Delta ^*_1 + \Delta ^*_2,
 \end{eqnarray*} 
with 
\begin{eqnarray*} 
 \Delta_1 ^*   &=& \frac 1 Q \sum _{q=0} ^Q \int_{x^* + 2qx}^{x^* + (2q+1)x} \frac Q {Th}  \left \{ G(t) - \mathbb{E} [ G(t)]\right \} \mathrm d t \\ \Delta_2 ^*   &=& \frac 1 Q \sum _{q=0} ^Q \int_{x^* + (2q+1)x}^{x^* + (2q+2)x} \frac Q {Th}  \left \{ G(t) - \mathbb{E} [ G(t)]\right \} \mathrm d t,
 \end{eqnarray*} 
 where $x^* =t_0 - hT$ and where $x$ is a value that depends on $T$ and that we will choose below. Furthermore, $Q$ is the smallest integer larger than $hT/x - 1$. Note that $\Delta_1^*$ has the same distribution as $\Delta_1^+$ which is defined as $\Delta_1^*$ but with $G(t)$ replaced by $G^+(t)$. The function $G^+(t) $ is defined as $G(t)$ but with the counting process $N$ replaced by $N^+$. We also define $G^+_q(t) $ as $G^+(t)$ but with $N^+$ replaced by $M_q^{x^*,x}$. From Lemma \ref{lem3} with $a =A$ we get that with probability $\geq 1 - CQ \exp(-\rho^* (x-A))$  
\begin{eqnarray*} 
 \Delta_1 ^+   &=& \frac 1 Q \sum _{q=0} ^Q \int_{x^* + 2qx}^{x^* + (2q+1)x} \frac Q {Th}  \left \{ G^+(t) - \mathbb{E} [ G(t)]\right \} \mathrm d t \\    &=& \frac 1 Q \sum _{q=0} ^Q \eta _q
 \end{eqnarray*} 
 with $$\eta_q = \int_{x^* + (2q+1)x}^{x^* + (2q+2)x} \frac Q {Th}  \left \{ G_q^+(t) - \mathbb{E} [ G(t)]\right \} \mathrm d t.$$
 The variables $\eta_q$ are mean zero independent random $d \times d$ matrices.  Furthermore, they are bounded as follows. Denote the jump points of the components of $M_q^{x^*,x}$ by $t_1^q, t_2^q,...$. First we have that for $x^* + 2qx \leq t < x^* + (2q+1) x$,  $1 \leq j,j^\prime \leq d$
 \begin{eqnarray} \label{poisa} |G^+_{q,j,j^\prime} (t)| &\leq& C \sum_{k,l\geq 1} \| \psi_j(t-t^q_k)\| \| \psi_{j^\prime}(t-t^q_l)\|\\ \nonumber 
 &\leq & C J \sum _{k,l\geq 1} \sum _{r,r^\prime \geq 1}\mathbb I _{(|t-t^q_k- \tau^r_j| \leq C/J)}\mathbb I _{(|t-t^q_l- \tau^{r^\prime}_{j^\prime}| \leq C/J)},\end{eqnarray}
 where $\tau_j^1, \tau_j^2,...$ and $\tau_{j^\prime}^1, \tau_{j^\prime}^2,...$ are elements of the support of $\psi_j$ or $\psi_{j^\prime}$, respectively. The supports are finite unions of intervals with diameter less than $C/J$. The values are chosen such that exactly one value is taken from each interval. We now construct upper bounds for the number of jump points of the components of $M_q^{x^*,x}$. For this purpose first note that Lemma \ref{lem1} implies that $\mathbb E [\exp(\rho N_{[t-2A,t]}/2)] < 2C$. This implies that $\sup_{0\leq t \leq T} N_{[t-A,t]} < C \log T$ with probability tending to 1. Because $\nu^{(l)}_0$ and $\mu_0^{(l,m)}$ are bounded by Assumption (A1) for $1 \leq l,m\leq d$ we have that $\lambda ^{(l)}(t) \leq C \log T$ with probability tending to 1 for $C$ chosen large enough. With similar argumennts as above we now argue that a homogeneous Poisson proces $\tilde N^q_t$ can be constructed with intensity $\tilde \lambda ^{(m)} (t) \equiv
 C \log T$ such that all jumps of a component of $M_q^{x^*,x}$ are also jumps of the corresponding component of $\tilde N^q $, with probability tending to 1. Denote the jump points of the components of $\tilde N^q $ in the interval $[x^*+ 2qx-A, x^*+2qx+x)$ by $\tilde t_1^q, \tilde t_2^q,...$.
 
Using this argument and \eqref{poisa} we get  that for $ \eta^+_{q;j,j^\prime} = \frac Q {Th}   \int_{x^* + 2qx}^{x^* + (2q+1)x}  G^*_{q,j,j^\prime}(t)\mathrm d t$ 

 \begin{eqnarray} \label{poisb} |\eta^+_{q;j,j^\prime} | &\leq & \frac C {x}   \int_{x^* + 2qx}^{x^* + (2q+1)x} J |G^+_{q,j,j^\prime}(t)|\mathrm d t \\ \nonumber
 &\leq & C  \frac J x
  \int_{x^* + 2qx}^{x^* + (2q+1)x} 
 \sum _{k,l\geq 1} \sum _{r,r^\prime \geq 1}\mathbb I _{(|t-\tilde t^q_k- \tau^r_j| \leq C/J)}\mathbb I _{(|t-\tilde t^q_l- \tau^{r^\prime}_{j^\prime}| \leq C/J)}
 \mathrm d t \\ \nonumber
 &\leq & C  \frac J x
  \int_{x^* + 2qx}^{x^* + (2q+1)x} 
  \sum _{k,l\geq 1} \sum _{r,r^\prime \geq 1}\mathbb I _{(|t-t^q_k- \tau^r_j| \leq C/J)}\mathbb I _{(|t-t^q_l- \tau^{r^\prime}_{j^\prime}| \leq C/J)} \mathrm d t 
 \\ \nonumber
 &\leq & C  \frac J x
  \int_{x^* + 2qx}^{x^* + (2q+1)x} 
 \left ( \sum _{k\geq 1} \sum _{r \geq 1}\mathbb I _{(|t-\tilde t^q_k- \tau^r_j| \leq C/J)}\right )^2  \mathrm d t 
  \\ \nonumber
  && \qquad + C  \frac J x
  \int_{x^* + 2qx}^{x^* + (2q+1)x} 
 \left (   \sum _{l\geq 1}\sum _{r^\prime \geq 1} \mathbb I _{(|t-\tilde t^q_l- \tau^{r^\prime}_{j^\prime}| \leq C/J)}
\right )^2  \mathrm d t 
 .\end{eqnarray}
 The first term on the right hand side of \eqref{poisb} can be bounded by 
 \begin{eqnarray*} 
 &\leq & C \frac J x \sum _{i=1} ^I
  \int_{x^* + 2qx+C(i-1)/J}^{x^* + 2qx+ Ci/J} 
 \left ( \sum _{k\geq 1} \sum _{r \geq 1} \mathbb I _{(|\tilde t^q_k+ \tau^r_j-x^* - 2qx- Ci/J| \leq 2C/J)}\right )^2  \mathrm d t \\
  &\leq & C \frac 1 x \sum _{i=1} ^I Z_{q,i}^2 \end{eqnarray*}
 with $ Z_{q,i}=
  \sum _{k\geq 1} \sum _{r \geq 1}\mathbb I _{(|\tilde t^q_k+ \tau^r_j-x^* - 2qx- Ci/J| \leq 2C/J)} $ and where $I$ is of order $xJ$.
  Note that for $1 \leq q \leq Q$ the variables $Z_{q,1}, Z_{q,3}, Z_{q,5}, ...$ and the variables $Z_{q,2}, Z_{q,4}, Z_{q,6}, ...$ are independent Poisson random variables with parameter of order $\log T/J$. We now use that with $X_j=(Z_{q,2j}^2 - \mathbb E[Z_{q,2j}^2])$ and with constants $C_k$ depending on $k$ we have that $\mathbb E [ X_j^{2k}] \leq C_k \log T / J$. By application of Rosenthals inequality, see \cite{R1970}, we have that $\mathbb E [(X_1+...+X_{I/2})^{2k}] \leq C_k \{I \mathbb E [X_1^{2k}] + (I \mathbb E [X_1^{2}])^{k}\} \leq  C_k (I \log T / J)^k$. Here we assume that $I$ is even. This gives that $\mathbb P (x^{-1} (X_1+...+X_{I/2}) > v) \leq  C_k (\log T/x)^k v^{-2k}$. 
    Because of $|\frac 1 x \sum _{i=1} ^I\mathbb E  [Z_{q,i}^2]| \leq C$ for  $C>0$ large enough we get that
    
    \begin{eqnarray} \label{poisc} &&\mathbb P \left ( |\eta^+_{q;j,j^\prime} | \geq C+v \mbox { for some } 1 \leq q \leq Q, 1 \leq j,j^\prime \leq d \right )\\
    && \qquad  \leq C\left ( \frac {\log T } x \right )^k Q J^2 v ^{-2k} \nonumber \end{eqnarray}
    for $v >0$.

We will use these considerations for the proof of the following lemma.

\begin{lemma} \label{lem7}
Make the assumptions of Theorem \ref{theo1}. Choose  $\epsilon > 0$ small enough such that $Q$ converges to infinity where $Q$ is chosen as the smallest integer larger than $hT^{1-\epsilon} - 1$. Then for $1 \leq j,j^\prime \leq d, 1 \leq q \leq Q$ there exist independent mean zero variables $\tilde \eta_{q,j,j^\prime}$ and $\tilde \eta_{q,j}$ with
\begin{eqnarray*}|\tilde \eta_{q,j,j^\prime}| &\leq& C T^\epsilon, \\
|\tilde \eta_{q,j}| &\leq& C T^\epsilon
\end{eqnarray*}
for some $C>0$ such that for all $\kappa>0$ there exists $C^*>0$ with 
\begin{eqnarray*}
\left | \Delta_{j,j^\prime}^+- \mathbb E[\Delta_{j,j^\prime}^+] - \frac 1 Q \sum_{q=1}^Q \tilde \eta _{q,j,j^\prime}\right | &\leq C^* T^{-\kappa},\\
\left | \delta_{j}^+  - \mathbb E[ \delta_{j}^+] - \frac 1 Q \sum_{q=1}^Q \tilde \eta _{q,j}\right | &\leq C^* T^{-\kappa}
\end{eqnarray*}
with probability $\geq 1- C^*T^{-\kappa}$. Here, $\Delta^+$ is a $d\times d$ random matrix and $\delta^+$ is a $\mathbb R^d$-valued random variable that have the same distribution as $\Delta$ or $\delta$, respectively. 
\end{lemma} 
\begin{proof}
Choose $\Delta^+$ as $\Delta$ but with $G(t)$ replaced by $G^+(t)$. With this choice the lemma holds. This can be seen by the considerations made above with the choice
$$ \tilde \eta_{q,j,j^\prime} = \eta_{q,j,j^\prime}^* \mathbb I_{[| \eta_{q,j,j^\prime}^* | \leq C T^\epsilon]} - \mathbb E \left [\eta_{q,j,j^\prime}^* \mathbb I_{[| \eta_{q,j,j^\prime}^* | \leq CT^\epsilon]}\right ].$$
This follows by application of \eqref{poisc} with $x = T^\epsilon$ and with $k>0$ large enough. In particular, one gets from \eqref{poisc} with $k>0$ large enough
 that 
$$
\mathbb E \left [\eta_{q,j,j^\prime}^* \mathbb I_{[| \eta_{q,j,j^\prime}^* | \leq C T^\epsilon]}\right ] \leq C^* T^{-\kappa}$$for all $\kappa >0$ with $C,C^*>0$ large enough. This shows the statements on $\Delta^+$ in the lemma. The statements for $\delta^+$ can be shown by similar arguments for an appropriately chosen $\delta^+$.
\end{proof} 
 \begin{lemma} \label{lem8}
Make the assumptions of Theorem \ref{theo1}. Then it holds uniformly for $\nu \in \mathbb R^d$ and $b_1,...,b_J \in \mathbb R$ with  $g= \sum_{j=1}^J b_j \psi_j$  
\begin{eqnarray*}
&&\left(\begin{array}{c}\nu \\b_1 \\ \vdots \\b_J\end{array}\right)^\intercal \Delta \left(\begin{array}{c}\nu \\b_1 \\\vdots \\b_J\end{array}\right) \\
&&\qquad=\frac 1 {Th} \int_{t_0 - Th}^{t_0 + Th} \left (\nu + \int_{t-A}^{t-}g^\intercal(t-u)   \mbox{d} N_u \right )^2 K\left ( \frac {t-t_0} {Th}\right )\mbox{d} t\\
&&\qquad \geq c \left (\|\nu\|^2 + \int_{0}^{A}g^\intercal(t)g(t) \mbox{d} t\right )
\end{eqnarray*}
with probability tending to one for $c> 0$ small enough. In particular, we have that $$\|\Delta ^{-1} x \| \leq c^{-1} \| x \|$$
for all $x\in \mathbb R ^{J+1}$ with probability tending to one.
\end{lemma} 
 \begin{proof}
 By application of Bernstein's inequality we get that for all $\epsilon, \kappa^* >0$ with $C,C^*>0$ large enough that for $1 \leq j, j^\prime \leq J$
 \begin{eqnarray*}
\mathbb P\left ( \left |  \frac 1 Q \sum_{q=1}^Q \tilde \eta _{q,j,j^\prime}\right | > C T^\epsilon \frac 1 {\sqrt {hT}} \right )&\leq& C^* T^{-\kappa^*},\end{eqnarray*}
   \begin{eqnarray*}
\mathbb P\left ( \left |  \frac 1 Q \sum_{q=1}^Q \tilde \eta _{q,j}\right | > C T^\epsilon \frac 1 {\sqrt {hT}} \right )&\leq& C^* T^{-\kappa^*}.\end{eqnarray*}
 By application of Lemma \ref{lem7} this gives with probability $\geq 1 - J^2 C^* T^{-\kappa^*}$ that 
 $$ \max_{1 \leq j, j^\prime \leq J} \left | \Delta_{j,j^\prime} - \mathbb{E}[ \Delta_{j,j^\prime}] \right | \leq C T^\epsilon \frac 1 {\sqrt{hT}}$$
 and thus  
$$ \|\Delta - \mathbb{E}[ \Delta]\|_2 \leq CJT^\epsilon  \frac 1 {\sqrt{hT}}.$$
By Assumption (A4) and Lemma \ref{lem6} we get the result of the lemma.\end{proof}

We now decompose $\hat \tau_1$ and $\hat \tau_2$ as follows
\begin{eqnarray*}
\hat \tau_1 &=& \hat \tau_{1,A} + \hat \tau_{1,B}+ \hat \tau_{1,C}, \\
\hat \tau_2 &=& \hat \tau_{2,A} + \hat \tau_{2,B}+ \hat \tau_{2,C},
\end{eqnarray*}
where 
 \begin{eqnarray*}
\hat \tau_{1,A} &=& \frac 1 {Th} \int_{t_0 - Th}^{t_0 + Th}  K\left ( \frac {t-t_0} {Th}\right ) \mbox{d} M^{(1)}_t, \\
\hat \tau_{1,B} &=& \frac 1 {Th} \int_{t_0 - Th}^{t_0 + Th}  K\left ( \frac {t-t_0} {Th}\right )\bigg  [\{ \nu_0(t/T) - \beta_0^*\} \\ &&\qquad + \int_{t-A}^{t-} \bigg \{\mu_0(t-s, t/T) - \sum_{j=1} ^J \beta_j^* \psi^\intercal_j(t-s) \bigg \} \mbox{d} N_s \bigg ] \mathrm d t ,
 \\
 \hat \tau_{1,C} &=& \frac 1 {Th} \int_{t_0 - Th}^{t_0 + Th}  K\left ( \frac {t-t_0} {Th}\right )\bigg  [\beta_0^*  \\ &&\qquad + \int_{t-A}^{t-}  \sum_{j=1} ^J \beta_j^* \psi^\intercal_j(t-s)  \mbox{d} N_s \bigg ] \mathrm d t ,
\\
\hat \tau_{2,A,j} &=&\frac 1 {Th} \int_{t_0 - Th}^{t_0 + Th}\int_{t-A}^{t-} \psi^\intercal_j(t-u)  K\left ( \frac {t-t_0} {Th}\right ) \mbox{d} N_u  \mbox{d} M_t^{(1)},\end{eqnarray*}
   \begin{eqnarray*}
\hat \tau_{2,B,j} &=&\frac 1 {Th} \int_{t_0 - Th}^{t_0 + Th}\int_{t-A}^{t-} \psi^\intercal_j(t-u)  K\left ( \frac {t-t_0} {Th}\right ) \mbox{d} N_u  \bigg  [\{ \nu_0(t/T) - \beta_0^*\} \\ &&\qquad + \int_{t-A}^{t-} \bigg \{\mu_0(t-s, t/T) - \sum_{j=1} ^J \beta_j^* \psi^\intercal_j(t-s) \bigg \} \mbox{d} N_s \bigg ] \mathrm d t ,\\
\hat \tau_{2,C,j} &=&\frac 1 {Th} \int_{t_0 - Th}^{t_0 + Th}\int_{t-A}^{t-} \psi^\intercal_j(t-u)  K\left ( \frac {t-t_0} {Th}\right ) \mbox{d} N_u \bigg  [\beta_0^*  \\ &&\qquad + \int_{t-A}^{t-}  \sum_{j=1} ^J \beta_j^* \psi^\intercal_j(t-s)  \mbox{d} N_s \bigg ] \mathrm d t 
\end{eqnarray*}
with $\beta_j^* = \beta_{j,0}^*$ for $j=0,...,J$, see (A3).

Note that $(\beta_0^*,...,\beta_J^*) ^\intercal  = \Delta^{-1}  (\hat \tau_{1,C}, \hat \tau_{2,C,1}, ...,\hat \tau_{2,C,J})^\intercal$ and that, according to (A3), we have that   $|\beta_0^* - \nu^{*}_0| \leq \varepsilon_T$ and 
 and  $| \mu_0^{*,(l)}(u) -  \sum_{j=1} ^J \beta_j^* \psi_j(u)| \leq \varepsilon_T$. This remark shows that for the proof of  Theorem \ref{theo1} it remains to show the following lemma.
 \begin{lemma} \label{lem9} Make the assumptions of Theorem \ref{theo1}. Then it holds that 
\begin{eqnarray}
\label{tau1a} 
\hat \tau_{1,A}  &=& O_P\left ( \sqrt{\frac 1 {hT}}\right ),\\ \label{tau1b}
\hat \tau_{1,B}&=& O_P\left ( \varepsilon_T\right ), \\ \label{tau2a}
\hat \tau_{2,A,j}  &=& O_P\left ( \sqrt{\frac 1 {hT}}\right ),\\ \label{tau2b}
\hat \tau_{2,B,j}&=& O_P\left (\varepsilon_T\right )
\end{eqnarray} for $j=1,...,J$.
\end{lemma}
\begin{proof}
Claims \eqref{tau1a} and \eqref{tau2a} follow by standard arguments. For the proof of \eqref{tau2b} one can use the inequality $|a b| \leq ((a^2/\varepsilon_T) + (b^2\varepsilon_T))/2$. Note that
\begin{eqnarray*}
&&\mathbb E \bigg [ \frac 1 {Th} \int_{t_0 - Th}^{t_0 + Th} K\left ( \frac {t-t_0} {Th}\right )   \bigg  [\{ \nu_0(t/T) - \beta_0^*\} \end{eqnarray*}
   \begin{eqnarray*} &&\qquad + \int_{t-A}^{t-} \bigg \{\mu_0(t-s, t/T) - \sum_{j=1} ^J \beta_j^* \psi^\intercal_j(t-s) \bigg \} \mbox{d} N_s \bigg ]^2 \mathrm d t  \bigg ] = O(\varepsilon_T^2),\\
&&\mathbb E \bigg [ \frac 1 {Th} \int_{t_0 - Th}^{t_0 + Th} K\left ( \frac {t-t_0} {Th}\right )\bigg [\int_{t-A}^{t-} \psi^\intercal_j(t-u)   \mbox{d} N_u  \bigg  ]^2 \mathrm d t  \bigg ] = O(1).
\end{eqnarray*}
Claims \eqref{tau1b} can be shown by similar arguments. 
\end{proof}

\end{document}